\documentclass[12pt]{amsart}
\usepackage{amssymb,amstext,amsmath,amscd,amsthm,amsfonts,ascmac,color,enumerate,geometry,graphicx,latexsym,multicol,stmaryrd}
\usepackage[all]{xy}
\usepackage{comment}
\newtheorem{thm}{Theorem}[section]
\newtheorem*{thm*}{Theorem}

\newtheorem{cor}[thm]{Corollary}
\newtheorem{lem}[thm]{Lemma}
\newtheorem{prop}[thm]{Proposition}
\theoremstyle{definition}

\newtheorem{dfn}[thm]{Definition}
\newtheorem*{dfn*}{Definition}
\newtheorem{rem}[thm]{Remark}

\newtheorem*{conj*}{Conjecture}
\newtheorem{ex}[thm]{Example}
\newtheorem*{ex*}{Example 4.11}
\newtheorem{nota}[thm]{Notation}
\theoremstyle{remark}
\newtheorem*{conv*}{Convention}

\newtheorem*{claim*}{Claim}

\numberwithin{equation}{thm}
\makeatletter
\def\mapstofill@{%
   \arrowfill@{\mapstochar\relbar}\relbar\rightarrow}
\newcommand*\xmapsto[2][]{%
   \ext@arrow 0395\mapstofill@{#1}{#2}}
\makeatother

\def\cE{\mathcal{E}}

\def\cO{\mathcal{O}}
\def\cS{\mathcal{S}}
\def\cT{\mathcal{T}}
\def\cX{\mathcal{X}}
\def\cY{\mathcal{Y}}
\def\cZ{\mathcal{Z}}

\def\fm{\mathfrak{m}}
\def\fn{\mathfrak{n}}
\def\fp{\mathfrak{p}}
\def\fq{\mathfrak{q}}

\def\rG{\mathrm{G}}
\def\rH{\mathrm{H}}
\def\rK{\mathrm{K}}
\def\rV{\mathrm{V}}

\def\LL{\mathbb{L}}
\def\RR{\mathbb{R}}
\def\ZZ{\mathbb{Z}}

\def\chara{\operatorname{char}}

\def\coh{\operatorname{\mathsf{coh}}}

\def\db{\operatorname{\mathsf{D^b}}}
\def\dbm{\operatorname{\mathsf{D^b_{fl}}}}
\def\dpf{\operatorname{\mathsf{D^{perf}}}}
\def\dpfm{\operatorname{\mathsf{D^{perf}_{fl}}}}

\def\e{\mathrm{e}}
\def\edim{\operatorname{edim}}
\def\Ext{\mathrm{Ext}}
\def\h{\mathbf{h}}
\def\height{\operatorname{\mathsf{ht}}}
\def\Hom{\operatorname{\mathsf{Hom}}}
\def\l{\ell}
\def\loc{\mathrm{loc}}
\def\ltensor{\otimes^{\mathbb{L}}}
\def\Min{\operatorname{\mathsf{Min}}}
\def\mod{\operatorname{\mathsf{mod}}}

\def\pol{\mathrm{pol}}

\def\Spec{\operatorname{\mathsf{Spec}}}
\def\thick{\operatorname{\mathsf{thick}}}

\def\ul{\underline}
\begin{document}
\title{On the categorical entropy of the Frobenius pushforward functor}
\author{Hiroki Matsui}
\address[H. Matsui]{Department of Mathematical Sciences, Faculty of Science and Technology, Tokushima University, 2-1 Minamijyousanjima-cho, Tokushima 770-8506, Japan}
\email{hmatsui@tokushima-u.ac.jp}

\author{Ryo Takahashi}
\address[R. Takahashi]{Graduate School of Mathematics, Nagoya University, Furocho, Chikusaku, Nagoya 464-8602, Japan}
\email{takahashi@math.nagoya-u.ac.jp}
\urladdr{http://www.math.nagoya-u.ac.jp/~takahashi/}

\thanks{2020 {\em Mathematics Subject Classification.} 13A35, 13D09, 14F08}
\thanks{{\em Key words and phrases.} derived category, categorical entropy, local entropy, Frobenius pushforward}
\thanks{The first author was partly supported by JSPS Grant-in-Aid for Early-Career Scientists 22K13894. The second author was partly supported by JSPS Grant-in-Aid for Scientific Research 19K03443.}

\begin{abstract}
In this paper, we consider the Frobenius pushforward endofunctor $F_\ast$ of the bounded derived category of finitely generated modules over an $F$-finite noetherian local ring.
We completely determine the categorical entropy of $F_\ast$ in the sense of Dimitrov, Haiden, Katzarkov, and Kontsevich.
\end{abstract}
\maketitle
\section{Introduction}

For a {\it categorical dynamical system}, namely, a pair $(\cT, \Phi)$ of a triangulated category $\cT$ and an exact endofunctor $\Phi:\cT \to \cT$, Dimitrov, Haiden, Katzarkov, and Kontsevich \cite{DHKK} have introduced an invariant $\h_t^{\cT}(\Phi)$ called the {\it categorical entropy} of $\Phi$, which is a categorical analogue of the topological entropy.
The categorical entropy $\h_t^{\cT}(\Phi)$ is a function in one real variable with values in $\RR \cup \{-\infty\}$ and measures the complexity of the exact endofunctor $\Phi$.

In this paper, we consider the Frobenius endomorphism $F:R \to R$ of a commutative noetherian local ring $R$ with prime characteristic $p>0$, assuming that $R$ is {\em $F$-finite}, that is to say, the map $F$ is (module-)finite. 
The Frobenius endomorphism $F$ induces two exact endofunctors.
One is called the {\it Frobenius pushforward} $F_*$ on the bounded derived category $\db(R)$ of finitely generated $R$-modules and the other is called the {\it Frobenius pullback} $\LL F^*$ on the derived category $\dpf(R)$ of perfect $R$-complexes. 
As to the latter, Majidi-Zolbanin and Miasnikov \cite{MZM} considered the full subcategory $\dpfm(R)$ of $\dpf(R)$ consisting of perfect complexes with finite length cohomologies, and computed the categorical entropy $\h_t^{\dpfm(R)}(\LL F^*)$.

The aim of this paper is to study the Frobenius pushforward $F_*$ on $\db(R)$ and compute its categorical entropy.
The main result of this paper is the following theorem.

\begin{thm}[Corollary \ref{3}]\label{main}
Let $(R,\fm,k)$ be a $d$-dimensional $F$-finite noetherian local ring with prime characteristic $p$.
Then there is an equality 
$$
\h_t^{\db(R)}(F_*) = d\log p + \log[F_*(k):k].
$$
\end{thm}

For an arbitrary finite local endomorphism $\phi:R \to R$ of an arbitrary noetherian local ring $R$, one can take its pushforward $\phi_*:\db(R) \to \db(R)$.
In such a general setting, one can still obtain the following weaker result on the categorical entropy.

\begin{thm}[Theorem \ref{lbd}]\label{main2}
Let $(R,\fm,k)$ be a $d$-dimensional noetherian local ring and $\phi:R \to R$ a local ring endomorphism of finite length. 
Then there is an inequality 
$$
\h_t^{\db(R)}(\phi_*) \ge\h_{\loc}(\phi) + \log [\phi_*(k):k].
$$
\end{thm}
\noindent
Here, $\h_{\loc}(\phi)$ is the {\em local entropy} of $\phi$, which has been introduced by Majidi-Zolbanin, Miasnikov, and Szpiro \cite{MZMS}.

The organization of this paper is as follows.
Section 2 is devoted to giving basic definitions including that of the categorical entropy.
In Section 3,  as a consequence of Theorem \ref{main2}, we prove one inequality of the equality given in Theorem \ref{main}.
In Section 4, we prove the opposite inequality, so that the the proof of Theorem \ref{main} is completed.

\begin{conv*}
Throughout the present paper, we assume that all rings are commutative and noetherian, and all subcategories are strictly full.
\end{conv*}

\section{Preliminaries}

In this section, we recall the notions of the {\it categorical entropy} of an exact endofunctor of a triangulated category and the {\it local entropy} of an endomorphism of local ring.
First of all, let us fix some notations.

\begin{nota}
\begin{enumerate}[\rm(1)]
\item	
For a ring $R$, denote by $\db(R)= \db(\mod R)$ the category of bounded complexes of finitely generated $R$-modules, and by $\dpf(R)$ the full subcategory of perfect complexes.
Here, a complex $X$ of $R$-modules is called {\it perfect} if there exists a bounded complex $P$ of finitely generated projective $R$-modules such that $X\cong P$ in $\db(R)$.
\item
Let $f: R \to S$ be a ring homomorphism.
\begin{enumerate}[\rm(a)]
\item
The {\it pullback functor} $f^*: \mod R \to \mod S$ is defined by $f^*(M) = M \otimes_R S$ for $M \in \mod R$.
The left derived functor of $f^*$ is an exact functor $\LL f^*: \dpf(R) \to \dpf(S)$.
\item
If $f$ is finite, then the {\it pushforward functor} $f_*: \mod S \to \mod R$ is defined by the abelian group $f_*(M) := M$ for $M \in \mod S$ together with the $R$-module structure via $f$.
The pushforward functor $f_*$ is exact, whence its right derived functor $f_* = \RR f_*: \db(\mod S) \to \db(\mod R)$ is defined by the degreewise application of $f_*$.
\end{enumerate}
\item
For a triangulated category $\cT$ and an object $X$ of $\cT$, we denote by $\thick X$ the smallest thick subcategory of $\cT$ that contains $X$.
\end{enumerate}
\end{nota}

\subsection{Categorical entropy}
We recall the definition and basic properties of the categorical entropy of an exact endofunctor of a triangulated category.
The following notation is useful.

\begin{nota}
Let $\cT$ be a triangulated category.
\begin{enumerate}[\rm(1)]
\item
For two full subcategories $\cX, \cY \subseteq \cT$, we denote by $\cX * \cY$ the full subcategory of $\cT$ consisting of objects $Z$ which fit into an exact triangle $X \to Z \to Y \to X[1]$ with $X \in \cX$ and $Y \in \cY$.
By the octahedral axiom, this symbol turns out to be associative: the equality $(\cX*\cY)*\cZ = \cX*(\cY*\cZ)$ holds for full subcategories $\cX,\cY,\cZ$ of $\cT$.
\item
For an integer $r \ge 2$ and full subcategories $\cE_1, \ldots, \cE_r \subseteq \cT$, we inductively define $\cE_1 *\cdots* \cE_r=(\cE_1* \cdots * \cE_{r-1}) * \cE_r$.
If $\cE_i = \{X_i\}$ for all $i$, then we write $X_1 * \cdots * X_r$ for $\cE_1 *\cdots* \cE_r$.
If $X_i=X$ for all $i$, then we put $X^{* r} = \overbrace{X * \cdots * X}^{r\text{ times}}$.
\end{enumerate}
\end{nota}

\begin{dfn}(\cite[Definition 2.1]{DHKK})
Let $X, Y$ be objects in $\cT$.
For a real number $t$, define the {\it complexity} $\delta_t^{\cT}(X,Y) \in \RR_{\ge 0} \cup\{\infty\}$ of $Y$ with respect to $X$ by
$$
\delta_t^{\cT}(X,Y) = \inf\left\{\sum_{i=1}^r e^{n_i t}\,\middle|\,\begin{matrix}Y \oplus Y' \in X[n_1] * \cdots*X[n_r]\\
\mbox{ for some $Y' \in \cT$ and $n_1, \ldots, n_r \in \ZZ$}\end{matrix}\right\}.
$$	
We will drop the superscript $\cT$ when there is no possibility of confusion.
By definition, $\delta_t(X, Y) <\infty$ if and only if $Y \in \thick X$.
\end{dfn}

We list several fundamental properties of the complexity.

\begin{lem}\label{compl}
Let $\cT$ be a triangulated category.
\begin{enumerate}[\rm(1)]
\item
For $X,Y,Z \in \cT$ with $Z \in \thick Y \subseteq \thick X$, one has $\delta_t(X, Z) \le \delta_t(X,Y)\delta_t(Y,Z)$.
\item
For $X,Y,Z \in \cT$, one has $\delta_t(X,Y) \le \delta_t(X, Y \oplus Z) \le \delta_t(X, Y) + \delta_t(X, Z)$.
\item
For $X,Y,Z \in \cT$, one has $\delta_t(X \oplus Y, Z) \le \delta_t(X, Z)$.
\item
For $X,Y \in \cT$, one has $\delta_t(X, Y[n]) = \delta_t(X, Y)e^{nt}$.
\item
For $X,Y, Y_1,\ldots, Y_r \in \cT$ with $Y \in Y_1 * \cdots * Y_r$, one has $\delta_t(X, Y) \le \sum_{i=1}^r \delta_t(X, Y_i)$.
\item
For an exact functor $\Phi: \cT \to \cT'$ and $X,Y \in \cT$, one has $\delta_t^{\cT'}(\Phi(X), \Phi(Y)) \le \delta_t^{\cT}(X, Y)$.
\end{enumerate}
\end{lem}

\begin{proof}
Assertions (1), (2) and (6) are shown in \cite[Proposition 2.2]{DHKK}, while (4) and (5) are direct consequences of the definition.
Let us prove (3). 
We may assume $\delta_t(X,Z)<\infty$.
Take $Z' \in \cT$ and $n_1, \ldots n_r\in\ZZ$ such that $Z \oplus Z' \in X[n_1] * X[n_2] * \cdots * X[n_r]$.
We easily see that
$$
Z \oplus Z' \oplus Y[n_1+n_2+\cdots +n_r] \in (X \oplus Y)[n_1] * (X \oplus Y)[n_2] * \cdots * (X \oplus Y)[n_r]
$$ 
holds. Therefore, the inequality $\delta_t(X \oplus Y, Z) \le \delta_t(X, Z)$ follows.
\end{proof}

An object $G$ of a triangulated category $\cT$ is called a {\it split generator} if $\cT = \thick G$.
For an excellent scheme $X$, the derived category $\db(\coh X)$ of bounded complexes of coherent sheaves on $X$ has a split generator by \cite[Theorem 4.15]{ELS}; see also \cite[Theorem 1.1]{IT2}.

\begin{dfn}[{\cite[Definition 2.4]{DHKK}}]\label{1}
Let $\cT$ be a triangulated category with a split generator $G$.
Let $\Phi:\cT \to \cT$ be an exact endofunctor.
For a real number $t$, we put
$$
\h_t^{\cT}(\Phi)=
\begin{cases}
\displaystyle\lim_{n \to \infty}\dfrac{1}{n} \log \delta_t^{\cT}(G, \Phi^n(G)) &\text{if $\delta_t(G,\Phi^e(G))\neq0$ for all $e \gg 0$,}\\
-\infty&\text{otherwise,}
\end{cases}
$$
and call it the {\it categorical entropy} of $\Phi$.
It follows from \cite[Lemma 2.6]{DHKK} that $\h_t^{\cT}(\Phi)$ exists in $[-\infty, \infty)$ and is independent of the choice of a split generator $G$.
Omitting the superscript $\cT$, we may simply write $\h_t(\Phi)$ if there is no danger of confusion.
\end{dfn}

\begin{rem}\label{rem}
Let $\cT$, $G$ and $\Phi$ be as in Definition \ref{1}.
\begin{enumerate}[\rm(1)]
\item
Let $n$ be an integer such that $\delta_t(G, \Phi^n(G)) = 0$.
Then it follows by Lemma \ref{compl}(1)(6) that
$$
\delta_t(G, \Phi^{n+1}(G)) \le \delta_t(G, \Phi^n(G)) \delta_t(\Phi^n(G),\Phi^{n+1}(G)) \le \delta_t(G, \Phi^n(G)) \delta_t(G,\Phi(G)) = 0,
$$
and this shows $\delta_t(G, \Phi^{n'}(G)) = 0$ for all $n' \ge n$.
\item
Let $G'$ be another split generator. 
Then, by Lemma \ref{compl}(1)(6), for each $n\ge0$ one has
\begin{align*}
\delta_t(G, \Phi^n(G))
&\le \delta_t(G, G') \delta_t(G', \Phi^n(G')) \delta_t(\Phi^n(G'), \Phi^n(G))\\
&\le \delta_t(G, G') \delta_t(G', \Phi^n(G')) \delta_t(G', G).
\end{align*}
Therefore, if $\delta_t(G, \Phi^n(G)) \neq 0$ for all $n \ge 0$, then $\delta_t(G', \Phi^n(G')) \neq 0$ for all $n \ge 0$, and moreover, $\delta_t(G, G')\ne0$ and $\delta_t(G', G) \neq 0$. 
\end{enumerate}
\end{rem}

Let us recall several asymptotic notations.
\begin{nota}
For two sequences $\{a_n\}_{n=1}^\infty$	 and $\{b_n\}_{n=1}^\infty$ of real numbers, we write
\begin{itemize}
\item
$a_n = O(b_n)$ if there is a real number $C>0$ such that $a_n \le C b_n$ for all $n \gg 1$, 
\item
$a_n = \Omega(b_n)$ if there is a real number $C>0$ such that $a_n \ge C b_n$ for all $n \gg 1$, and
\item
$a_n = \Theta(b_n)$ if $a_n = O(b_n)$ and $a_n = \Omega(b_n)$.
\end{itemize}
\end{nota}

We present some elementary facts about the asymptotic notations introduced above.

\begin{lem}
Let $\{a_n\}_{n=1}^\infty$ be a sequence of positive real numbers such that $a_{m+n} \le a_m a_n$ for all $m, n \ge 1$.
Let $u$ be a positive real number.
Set $\alpha=\lim_{n \to \infty} \frac{\log a_n}{n} \in [-\infty, \infty)$ (this limit exists by Fekete's lemma).
Then the following statements hold.
\begin{enumerate}[\rm(1)]
\item
\begin{enumerate}[\rm(a)]
\item
If $a_n = O(u^n)$, then $\alpha \le \log u$.
\item
If $a_n = \Omega(u^n)$, then $\alpha \ge \log u$.
\item	
If $a_n = \Theta(u^n)$, then $\alpha= \log u$.
\end{enumerate}
\item
For any real number $\beta > \alpha$, one has $a_n = O(e^{n \beta})$.
\end{enumerate}
\end{lem}

\begin{proof}
(1) Let us show assertion (a).
By definition, there is a real number $C > 0$ such that $a_n/u^n \le C$ for all $n \gg 1$.
It holds that
\begin{align*}
\textstyle0 = \limsup_{n \to \infty}(\tfrac{1}{n} \log C)
&\textstyle\ge \limsup_{n \to \infty}(\frac{1}{n} \log\frac{a_n}{u^n})\\
&\textstyle= \limsup_{n \to \infty}(\frac{\log a_n}{n} - \log u) 
\textstyle= \lim_{n \to \infty}\frac{\log a_n}{n} - \log u,
\end{align*} 
which implies that $\alpha=\lim_{n \to \infty}\frac{\log a_n}{n} \le \log u$.
Assertions (b) and (c) can be shown similarly.

(2) Since $\lim_{n \to \infty}\frac{\log a_n}{n}< \beta$, we have $\lim_{n \to \infty}\log\frac{\sqrt[n]{a_n}}{e^{\beta}} =\lim_{n \to \infty}(\frac{\log a_n}{n}-\beta)< 0$.
This means that 
$
\lim_{n \to \infty}\frac{\sqrt[n]{a_n}}{e^{\beta}} < 1.
$
Hence the inequality
$
\frac{\sqrt[n]{a_n}}{e^{\beta}} < 1
$
holds for all $n \gg 1$, and we get $a_n < e^{n\beta}$ for all $n \gg 1$.
Now the conclusion $a_n = O(e^{n\beta})$ follows.
\end{proof}

Using the above lemma, we can prove the following proposition, which connect the order of the complexity and the categorical entropy.

\begin{prop}\label{cor2}
Let $\cT$ be a triangulated category.
Let $G$ be a split generator of $\cT$.
Let $\Phi:\cT \to \cT$ be an exact functor.
Then the following statements hold true.
\begin{enumerate}[\rm(1)]
\item
\begin{enumerate}[\rm(a)]
\item	
If $\delta_t(G, \Phi^n (G)) =O(u^n)$, then $\h_t(\Phi) \le \log u$.	
\item
If $\delta_t(G, \Phi^n (G)) =\Omega(u^n)$, then $\h_t(\Phi) \ge \log u$.	
\item
If $\delta_t(G, \Phi^n (G)) = \Theta(u^n)$, then $\h_t(\Phi) = \log u$.
\end{enumerate}
\item
For any $\beta > \h_t(\Phi)$, one has the equality $\delta_t(G, \Phi^n(G)) = O(e^{n \beta})$.	
\end{enumerate}
\end{prop}

\begin{rem}
Let $\cT,G,\Phi$ be as in Proposition \ref{cor2}.
The {\it categorical polynomial entropy}
$$
\h_t^{\pol}(\Phi) := \underset{n \to \infty}{\limsup}\, \dfrac{\log \delta_t(G, \Phi^n(G)) - n h_t(\Phi)}{\log n}
$$
has been introduced by Fan, Fu and Ouchi \cite{FFO}.
If $\h_t^{\pol}(\Phi)$ is positive, then it is easy to see that $\delta_t(G, \Phi^n(G)) \neq O(e^{n \h_t(\Phi)})$.
Therefore, for each $\alpha\in\RR$, the equality $\delta_t(G, \Phi^n(G)) = O(e^{n \alpha})$ is stronger than the equality $\h_t(\Phi) = \alpha$ in general.
\end{rem}

\subsection{Local entropy}

Next we recall the notion of the local entropy of a local dynamical system introduced in \cite{MZMS}.
We start by basic notions from (local) commutative algebra. 

\begin{dfn}
A local homomorphism $\phi: (R,\fm,k) \to (S,\fn,l)$ of local rings is of {\it finite length} if the ideal $\phi(\fm)S$ of $S$ is $\fn$-primary.
If $\phi$ is finite, then it is of finite length.
\end{dfn}

Here are examples of local endomorphisms of finite length which we consider in this paper.

\begin{ex}\label{endo}
\begin{enumerate}[\rm(1)]
\item
If a local ring $R$ has prime characteristic $p>0$, then the {\it Frobenius endomorphism} $F: R \to R$ given by $x \mapsto x^p$ is of finite length.
\item
Let $\Gamma= \sum_{i=1}^r \ZZ_{\ge 0}\,a_i$ be a finitely generated additive monoid and $k$ a field.
Let $k[\![\Gamma]\!]$ denote the completion of the monoid ring $k[\Gamma]= k[t^{a_i} \mid i=1,\ldots,r]$ with respect to its homogeneous maximal ideal $(t^{a_i} \mid i=1,\ldots,r)$. 
For an integer $m\ge 1$, the morphism of monoids $\Gamma \to \Gamma$ given by $x \mapsto mx$ induces a local ring endomorphism $F_m: k[\![\Gamma]\!] \to k[\![\Gamma]\!]$ which is of finite length. 
\end{enumerate}
\end{ex}

For a local endomorphism $\phi$ of a local ring $R$, an integer $e \ge 1$, and an $R$-module $M$, we  write ${}^{\phi^{e}}\!M$ for $\phi^e_*(M)$.
When $\phi$ is understood from the context, we simply write ${}^e\!M$.

\begin{dfn}
Let $\phi$ be a local endomorphism of a local ring $R$. 
We say that $R$ is {\it $\phi$-finite} provided that ${}^\phi\!R$ is finitely generated as an $R$-module.
\end{dfn}

\begin{rem}\label{2}
\begin{enumerate}[\rm(1)]
\item
If a local ring $R$ is $\phi$-finite, then so is the residue field $k$, i.e., $[{}^\phi k: k] < \infty$.
\item
Let $R$ be a local ring of prime characteristic $p$.
If $R$ is $F$-finite, then $R$ is excellent.
The converse holds if the residue field $k$ is $F$-finite; see \cite[Theorem 2.5 and Corollary 2.6]{Kunz}.
\item	
Let $k$ be a field.
Let $\Gamma= \sum_{i=1}^r \ZZ_{\ge 0}\,a_i$ be a finitely generated additive monoid.
Then the local ring $k[\![\Gamma]\!]$ is $F_m$-finite for every positive integer $m$.
Indeed, ${}^{F_m}k[\![\Gamma]\!]$ is generated by the monomials $\{t^{i_1a_1} \cdots t^{i_ra_r} \mid 0 \le i_1, \ldots, i_r < m\}$ as an $k[\![\Gamma]\!]$-module.
\end{enumerate}	
\end{rem}

The following easy lemma is frequently used later.

\begin{lem}\label{len}
Let $\phi$ be a finite local endomorphism of a local ring $R$.
For an $R$-module $M$ of finite length and $e \ge 0$, the $R$-module ${}^e\!M$ has finite length with $\l_R({}^e\!M) = [{}^1k:k]^e \cdot\l_R(M)$.	
\end{lem}

\begin{proof}
First we prove $\l_R({}^e\!M) = [{}^ek:k] \cdot\l_R(M)$ by induction on $n:=\l_R(M)$.
The case $n = 1$ is clear as $M\cong k$.
Let $n>1$.
Then there is an exact sequence $0 \to N \to M \to k \to 0$ of $R$-modules.
Applying ${}^e(-)$ to this sequence, we get an exact sequence $0 \to {}^e\! N \to {}^e\!M \to {}^ek \to 0$.
By the induction hypothesis, $\l_R({}^e\! N) = [{}^ek:k] \cdot \l_R(N)$ and $\l_R({}^ek) = [{}^ek:k]$, which yield
$$
\l_R({}^e\!M) =\l_R({}^e\! N) + \l_R({}^ek) = [{}^ek:k] \cdot \l_R(N) + [{}^ek:k] =[{}^ek:k] (\l_R(N) + 1) = [{}^ek:k] \cdot \l_R(M).
$$

Next we prove $[{}^ek:k] = [{}^1k:k]^e$.
Since ${}^1k \cong k^{\oplus [{}^1k:k]}$, one has ${}^e k = {}^{e-1}({}^1k) \cong {}^{e-1}(k^{\oplus [{}^1k:k]}) \cong ({}^{e-1}k)^{\oplus [{}^1k:k]}$.
We thus get an isomorphism ${}^ek \cong k^{\oplus [{}^1k:k]^e}$ by induction on $e$.
\end{proof}

\begin{dfn}[{\cite[Definition 5]{MZMS}}]
By a {\it local algebraic dynamical system} $(R, \phi)$, we mean a pair of a local ring $R$ and a local endomorphism $\phi$ of $R$ which is of finite length.
A local algebraic dynamical system $(R, \phi)$ is called {\it finite} if $R$ is $\phi$-finite. 

For a local algebraic dynamical system $(R, \phi)$, the {\it local entropy} of $\phi$ is defined by
$$
\h_{\loc}(\phi) := \underset{e\to \infty}{\lim}\, \dfrac{\log \l_R(R/\phi^e(\fm)R)}{e}.
$$	
This limit exists and is nonnegative by \cite[Theorem 1]{MZMS}.
\end{dfn}

For a local ring $R$ we denote by $\e(R)$ and $\edim R$ the (Hilbert--Samuel) multiplicity and the embedding dimension of a local ring $R$, respectively.
Under a certain assumption, we can explicitly compute the local entropy.

\begin{lem}\label{locent}
Let $(R,\phi)$ be a local algebraic dynamical system.
Put $d=\dim R$ and $\nu=\edim R$.
Assume that there exists an integer $u\ge1$ such that $\fm^{\nu u^e} \subseteq \phi^e(\fm)R \subseteq \fm^{u^e}$ for all $e \gg 1$. 
Then $\l_R(R/\phi^e(\fm)R) = \Theta(u^{de})$.
In particular, there is an equality $\h_{\loc}(\phi) = d\log u$.
\end{lem}

\begin{proof}
By assumption, we have inequalities $\l_R(R/\fm^{u^e}) \le \l_R(R/\phi^e(\fm)R) \le \l_R(R/\fm^{\nu u^e})$ for all $e \gg 1$.
Since $\l_R(R/\fm^{u^e})/u^{de}$ and $\l_R(R/\fm^{\nu u^e})/u^{de}$ converge to the nonzero real numbers $\e(R)d!$ and $\e(R)\nu d!$ respectively, we obtain the equality $\l_R(R/\phi^e(\fm)R) = \Theta(u^{de})$.
The equality $\h_{\loc}(\phi) = d\log u$ follows from $\l_R(R/\phi^e(\fm)R) = \Theta(u^{de})$ and Proposition \ref{cor2}(1c).
\end{proof}

Since the endomorphisms in Example \ref{endo} satisfy the assumption in this lemma, we get:

\begin{cor}
\begin{enumerate}[\rm(1)]
\item
Let $R$ be a $d$-dimensional $F$-finite local ring of characteristic $p$.
Then the equalities $\l_R(R/F_*^e(\fm)R) = \Theta(p^{de})$ and $\h_{\loc}(F) = d\log p$ hold.
\item
Let $k$ be a field, $\Gamma$ a finitely generated additive monoid and $R=k[\![\Gamma]\!]$.
Let $m \ge 1$ be an integer.
Then one has the equalities $\l_R(R/(F_m)_*^e(\fm)R) = \Theta(m^{de})$ and $\h_{\loc}(F_m) = d\log m$.
\end{enumerate}	
\end{cor}

In this paper, we are mainly concerned with the Frobenius functor. 
Let $(R, \fm, k)$ be an $F$-finite local ring with characteristic $p$.
Then the Frobenius endomorphism $F:R \to R$ induces the {\it Frobenius pushforward}
$$
\begin{array}{cccc}
F_* = \RR F_* : &\db(R) & \xrightarrow{\qquad \qquad} & \db(R) \\
&\rotatebox{90}{$\subseteq$} & & \rotatebox{90}{$\subseteq$}\\
&\dbm(R) & \xmapsto{\qquad  \qquad} & \dbm(R),
\end{array}
$$ 
and the {\it Frobenius pullback} (see \cite[Proposition 1.10]{MZM})
$$
\begin{array}{cccc}
\LL F^*: &\dpf(R) & \xrightarrow{\qquad \qquad} & \dpf(R) \\
&\rotatebox{90}{$\subseteq$} & & \rotatebox{90}{$\subseteq$}\\
&\dpfm(R) & \xmapsto{\qquad  \qquad} & \dpfm(R).
\end{array}
$$
Here, $\dbm(R),\dpfm(R)$ stand for the subcategories of $\db(R),\dpf(R)$ consisting of complexes with finite length homologies, respectively.
The categorical entropy of the Frobenius pullback on $\dpfm(R)$ is computed by Majidi-Zolbanin and Miasnikov:

\begin{thm}[{\cite[Corollary 2.6]{MZM}}]
Let $R$ be a $d$-dimensional complete local ring of prime characteristic $p>0$.
Then the equality $\h_t^{\dpfm(R)}(\LL F^*) = d\log p$ holds.
\end{thm}

On the other hand, for the Frobenius pullback on $\dpf(R)$, the following holds.

\begin{prop}
Let $R$ be a local ring of prime characteristic $p>0$.
Then one has the equality $\delta_t^{\dpf(R)}(R,\LL F^*(R)) = 1$.
In particular, the equality $\h_t^{\dpf(R)}(\LL F^*) = 0$ holds.
\end{prop}

\begin{proof}
Let us prove the first equality.
As $(\LL F^*)^e(R) \cong R$ for all $e \ge 0$, it suffices to show $\delta_t^{\dpf(R)}(R,R) = 1$.
The inequality $\delta_t^{\dpf(R)}(R,R) \le 1$ obviously holds.
Assume that there exist $n_1, \ldots, n_r\in\ZZ$ and $X \in \dpf(R)$ with $R \oplus X \in R[n_1] * \cdots * R[n_r]$.
Then at least one of the numbers $n_1, \ldots, n_r$ is zero.
Indeed, if $n_1, \ldots, n_r$ are all nonzero, then the equalities $\Hom_{\dpf(R)}(R, R[n_i]) = 0$ with $i=1,\ldots, r$ yield $\Hom_{\dpf(R)}(R, Y) = 0$ for any $Y \in R[n_1] * \cdots * R[n_r]$.
In particular, we have $\Hom_{\dpf(R)}(R, R \oplus X) = 0$, which leads a contradiction.
Thus $\sum_{i=1}^r e^{n_it} \ge 1$, and so $\delta_t^{\dpf(R)}(R,R) \ge 1$.
Now the first equality of the proposition follows.
As $R$ is a split generator of $\dpf(R)$, the second equality follows from the first.
\end{proof}

We can also compute the categorical entropy of the Frobenius pushforward on $\dbm(R)$.

\begin{prop}
Let $(R,\fm,k)$ be a $d$-dimensional $F$-finite local ring.
Then for every $e\ge1$ the equality $\delta_t^{\dbm(R)}(k, {}^ek) = [{}^1k:k]^e$ holds.
In particular, one has $\h_t^{\dbm(R)}(F_*) = \log [{}^1k:k]$.
\end{prop}

\begin{proof}
As $k$ is a split generator of $\dbm(R)$, it is enough to show the first equality.
The inequality $\delta_t^{\dbm(R)}(k, {}^e k) \le [{}^1k:k]^e$ is trivial because ${}^e k \cong k^{\oplus [{}^1k:k]^e}$.
Take integers $n_1,\ldots, n_r$ and $X \in \dbm(R)$ such that $k^{\oplus [{}^1k:k]^e} \oplus X \in k[n_1] * \cdots * k[n_r]$.
Since $\l_R(\rH^0(-))$ is a subadditive function with respect to exact triangles, we obtain the (in)equalities
$$
\textstyle
[{}^1k:k]^e \le  \l_R(\rH^0(k^{\oplus [{}^1k:k]^e})) \le \l_R(\rH^0(k^{\oplus [{}^1k:k]^e} \oplus X)) \le \sum_{i=1}^r \l_R(\rH^0(k[n_i])) = \#\{i \mid n_i = 0\} .
$$
Hence $\sum_{i=1}^r e^{n_i t} \ge \sum_{n_i = 0} 1 \ge [{}^1k:k]^e$.
The inequality $\delta_t^{\dbm(R)}(k, {}^ek) \ge [{}^1k:k]^e$ follows.
\end{proof}

From these results, the remaining problem is to compute the categorical entropy of the Frobenius pushforward on $\db(R)$, which we shall deal with in the subsequent sections. 

\section{Lower bounds}

In this section, we give in terms of local entropies a lower bound on the categorical entropy $\h_t(\phi)$ of the pushforward $\phi_*: \db(R) \to \db(R)$ along a finite local endomorphism $\phi$ of a local ring $R$.
The following lemma plays a key role in showing the main theorem in this section, ideas of whose proof are taken from \cite[Lemma 2.1]{MZM}.
Denote by $\rK(-)$ the Koszul complex.

\begin{lem}\label{lem1}
Let $(R,\fm,k)$ be a local ring.
Let $\underline{x}$ be a sequence of elements of $R$ with $\sqrt{(\underline{x})} = \fm$.
Let $G \in \db(R)$.
Fix an integer $N$ such that $\rH^i(G \ltensor_R \rK(\underline{x})) = 0$ for all $|i| > N$, and set
$
B = \max \{ \l_R(\rH^i(G \ltensor_R \rK(\underline{x}))) \mid -N \le i \le N\}.
$ 
Then for any $E \in \db(R)$ and $m \in \ZZ$, one has
$$
\l_R(\rH^m(E\ltensor_R \rK(\underline{x}))) \le Be^{mt}e^{N|t|} \delta_t(G, E).
$$
\end{lem}

\begin{proof}
We can assume $E \in  \thick G$, because otherwise, the right-hand side is positive infinity.
We find $E' \in \db(R)$ and integers $n_1, \ldots, n_r$ such that
$
E \oplus E' \in G[n_1] * G[n_2] * \cdots * G[n_r].
$
Applying $- \ltensor_R \rK(\ul{x})$, we get a containment
$$
(E \ltensor_R \rK(\underline{x})) \oplus (E' \ltensor_R \rK(\underline{x})) \in (G\ltensor_R \rK(\underline{x}))[n_1] * (G\ltensor_R \rK(\underline{x})[n_2] * \cdots * (G\ltensor_R \rK(\underline{x})[n_r].
$$
Since $\l_R(\rH^i(-))$ is a subadditive function with respect to exact triangles, the inequalities
\begin{align*}
\l_R(\rH^m(E \ltensor_R \rK(\underline{x})))
&\le \l_R(\rH^m(E \ltensor_R \rK(\underline{x})) \oplus \rH^m(E' \ltensor_R \rK(\underline{x})))\\
&\textstyle\le  \sum_{i=1}^r \l_R(\rH^{m+n_i}(G \ltensor_R \rK(\underline{x}))) \le B |S_m|
\end{align*}
follow, where $S_m := \{i \mid -N \le m + n_i \le N\}$.
Using the inequality $e^x \ge e^{-|x|}$, we have
$$
\textstyle
e^{-N|t|} |S_m| \le \sum_{i \in S_m} e^{-|(m+n_i)t|} \le \sum_{i =1}^r e^{-|(m+n_i)t|} \le \sum_{i =1}^r e^{(m+n_i)t} = e^{mt} \sum_{i =1}^r e^{n_i t}.
$$
Thus $\l_R(\rH^m(E \ltensor_R \rK(\underline{x}))) \le B |S_m| \le Be^{mt}e^{N|t|}\sum_{i =1}^r e^{n_i t}$, and the assertion follows.
\end{proof}

The following theorem is the main result of this section.

\begin{thm}\label{lbd}
Let $(R,\fm,k)$ be a local ring with Krull dimension $d$ and embedding dimension $\nu$.
Let $(R,\phi)$ be a finite local algebraic dynamical system. 
Suppose that $\db(R)$ has a split generator (e.g., if $R$ is excellent).	
Then there is an inequality
$$
\h_t(\phi_*) \ge \h_{\loc}(\phi) + \log [{}^1k:k].
$$
If there exists an integer $u\ge1$ such that the inclusions $\fm^{\nu u^e} \subseteq \phi^e(\fm)R \subseteq \fm^{u^e}$ hold for all $e \gg 1$, then the following stronger equality holds for any split generator $G$ of $\db(R)$.
$$
\delta_t(G, {}^e G) = \Omega(([{}^1k:k]u^d)^e).
$$
\end{thm}

\begin{proof}
We begin with showing the first assertion.
Let $G$ be a split generator of $\db(R)$.
We may assume $\inf G:=\inf\{i\mid\rH^i(G)\ne0\}= 0$ and $R$ is a direct summand of $\rH^0(G)$ by Remark \ref{rem}(2).
By Lemma \ref{lem1}, for any fixed $t$ there exists $D_t>0$ such that
$
\l_R(\rH^0({}^eG \otimes \rK(\underline{x}))) \le D_t\cdot\delta_t(G, {}^e G),
$
where $\ul{x}$ is a system of generators of $\fm$.
As $\inf G = 0$, we get equalities
$$
\rH^0({}^eG \otimes \rK(\underline{x})) = \rH^0({}^eG)/\fm \rH^0({}^eG) =  {}^e\rH^0(G)/\fm ({}^e\rH^0(G)) = {}^e (\rH^0(G)/\phi^e(\fm)\rH^0(G)).
$$
Here, the second equality follows since the Frobenius pushforward $F_*$ is exact on the category of $R$-modules.
From these observations, we obtain (in)equalities
\begin{align}\tag{$*$} \label{eq}
\begin{aligned}
\delta_t(G, {}^eG) 
&\ge D_t^{-1} \cdot \l_R({}^e\!\left[\rH^0(G)/\phi^e(\fm)\rH^0(G)\right]) 
= D_t^{-1} \cdot [{}^ek:k] \cdot \l_R(\rH^0(G)/\phi^e(\fm)\rH^0(G))\\
&= D_t^{-1} \cdot [{}^1k:k]^e \cdot \l_R(\rH^0(G)/\phi^e(\fm)\rH^0(G))
 \ge D_t^{-1} \cdot [{}^1k:k]^e \cdot \l_R(R/\phi^e(\fm)R),
\end{aligned}
\end{align}
where for the first equality we use Lemma \ref{len} and for the last inequality we use the assumption that $R$ is a direct summand of $\rH^0(G)$.
Take the logarithms of both sides of \eqref{eq}, divide them by $e$, and take the limits to get the inequality $\h_t(\phi_*) \ge \h_{\loc}(\phi) + \log [{}^1k:k]$.

Finally, we show the second assertion of the theorem.
Lemma \ref{locent} implies $\l_R(R/\phi^e(\fm)R) = \Theta(u^{de})$.
It follows from Remark \ref{rem}(2) and \eqref{eq} that $\delta_t(G,{}^eG) = \Omega(([{}^1k:k]u^d)^e)$.
\end{proof}

As a direct consequence of Theorem \ref{lbd}, we get the following corollary.

\begin{cor}\label{corlbd}
\begin{enumerate}[\rm(1)]
\item	
Let $(R,\fm,k)$ be a $d$-dimensional $F$-finite local ring with characteristic $p$.
Then the equality $
\delta_t(G, {}^e G) = \Omega(([{}^1k:k]p^d)^e)
$
holds for every split generator $G$ of $\db(R)$.
In particular, there is an inequality $\h_t(F_*) \ge d\log p + \log [{}^1k:k]$.
\item
Let $k$ be a field, $\Gamma$ a finitely generated additive monoid and $R=k[\![\Gamma]\!]$.
Let $m$ be a positive integer.
Then one has the equality $\delta_t(G, (F_m)_*^e G) = \Omega(m^{de})$ for every split generator $G$ of $\db(R)$.
In particular, there is an inequality $\h_t((F_m)_*) \ge d\log m$.
\end{enumerate}
\end{cor}

Next, we generalize Corollary \ref{corlbd}(1) to the global case.
To this end, we need a couple of lemmas.
For a prime ideal $\fp$ of a ring $R$, set $\alpha_\fp = \log_p [{}^1k(\fp): k(\fp)]$.	

\begin{lem}\label{glem1}
Let $R$ be an $F$-finite ring.
Let $\fp\subseteq\fq$ be in $\Spec R$.
Then $\alpha_\fp + \height \fp = \alpha_\fq + \height \fq$.
\end{lem}

\begin{proof}
It follows from \cite[Proposition 2.3]{Kunz} that $[{}^1k(\fp):k(\fp)] = [{}^1k(\fq):k(\fq)] \cdot p^{\dim R_\fq/\fp R_\fq}$, where $p=\chara R$.
The ring $R$ is excellent (see Remark \ref{2}(2)), and in particular, it is catenary.
Therefore, we have the equalities $\alpha_\fp + \height \fp = \alpha_\fq + \dim R_\fq/\fp R_\fq + \height \fp = \alpha_\fq + \height \fq$.
\end{proof}

Modifying the definition of a Hochster--Huneke graph \cite{HH}, we introduce the following graph $\rG(R)$ associated to a noetherian ring $R$.
\begin{itemize}
\item
The set of vertices is $\Min R$, the set of minimal prime ideals of $R$.
\item
There is an edge between two prime ideals $\fp$ and $\fq$ if $\fp + \fq \neq R$. 	
\end{itemize}

\begin{lem}[cf. {\cite[Theorem 3.6]{HH}}]\label{glem2}
Let $R$ be a ring.
If $\Spec R$ is connected as a topological space, then $\rG(R)$ is connected as a graph.	
\end{lem}

\begin{proof}
Assume that $\rG(R)$ is not connected as a graph. 
Then there exists a nontrivial partition $\Min R = A \sqcup B$ such that $\fp + \fq = R$ for all $\fp \in A$ and $\fq \in B$.
Neither the ideal $I = \bigcap_{\fp \in A} \fp$ nor the ideal $J = \bigcap_{\fq \in B} \fq$ is nilpotent, while the ideal $IJ$ is nilpotent.
We have
$$
\textstyle
\rV(I) \cap \rV(J) = (\bigcup_{\fp \in A} \rV(\fp)) \cap (\bigcup_{\fq \in B} \rV(\fq))= \bigcup_{\fp \in A, \fq \in B} \rV(\fp + \fq) = \emptyset.
$$
We obtain a nontrivial decomposition $\Spec R = \rV(IJ) = \rV(I) \sqcup \rV(J)$ into disjoint closed subsets.
Therefore, $\Spec R$ is not connected as a topological space.
\end{proof}

Lemma \ref{glem2} says that for any two prime ideals $\fp, \fq$ of a ring $R$ with $\Spec R$ connected, there is a sequence of prime ideals $\fp_1=\fp, \fp_2, \ldots, \fp_n=\fq, \fq_1, \fq_2, \ldots, \fq_{n-1}$ such that $\fp_i, \fp_{i+1} \subseteq \fq_i$ for $i=1,2,\ldots,n-1$.
Lemma \ref{glem1} implies $\alpha_\fp + \height \fp = \alpha_\fq + \height \fq$ and hence the number $\alpha_\fp + \height \fp$ is constant for $\fp \in \Spec R$.
As a result, for an $F$-finite noetherian scheme $X$, the function
$$
X \to \ZZ, \,\, x \mapsto \dim \cO_{X,x} + \log_p [{}^1k(x): k(x)]
$$
is continuous. 
If $X$ is connected, then this is a constant function, namely, the number
$$
\beta_X :=  \dim \cO_{X,x} + \log_p [{}^1k(x): k(x)]
$$
is independent of the choice of $x \in X$.

\begin{cor}\label{lbd-sch}
Let $X$ be a $d$-dimensional $F$-finite connected noetherian scheme of characteristic $p$.
Then the equality $\delta_t(G, {}^eG) = \Omega(p^{e\beta_X})$ holds for any split generator $G$ of $\db(\coh X)$.
In particular, the following holds for any $x \in X$.
$$
\h_t^{\db(\coh X)}(F_*) \ge \beta_X\log p = \dim \cO_{X,x} \cdot \log p + \log[{}^1k(x):k(x)].
$$
\end{cor}

\begin{proof}
Since $X$ is $F$-finite, it is excellent by Remark \ref{2}(2).
It follows from \cite[Theorem 4.15]{ELS} that there is a split generator $G$ of $\db(\coh X)$.
Let $x$ be any point of $X$. 
Then $G_x$ is a split generator of $\db( \cO_{X,x})$. 
Indeed, there is an equivalence of triangulated categories
$$
\db(\coh X)/\cS(x) \cong \db( \cO_{X,x}),\,\, E \mapsto E_x,
$$
where $\cS(x) := \{E \in \db(\coh X) \mid E_x \cong 0 \}$; see \cite[Lemma 2.2]{Orl} and \cite[Lemma 3.2]{Mat}.
Note that there is a commutative diagram 
$$
\xymatrix{
\db(\coh X) \ar[r]^{F_*} \ar[d]_{(-)_x} & \db(\coh X) \ar[d]^{(-)_x}  \\
\db( \cO_{X,x}) \ar[r]_{F_*} & \db( \cO_{X,x}).
}
$$
Here, the vertical functors are the stalk functors and the horizontal ones are the Frobenius pushforwards.
Using \cite[Proposition 2.2(c)]{DHKK}, we get $\delta_t^{\db( \cO_{X,x})}(G_x, {}^eG_x) \le \delta_t^{\db(\coh X)}(G, {}^eG)$.
Corollary \ref{corlbd}(1) implies $\delta_t^{\db( \cO_{X,x})}(G_x, {}^eG_x) = \Omega(([{}^1k(x):k(x)]d^{\dim \cO_{X,x}})^e) = \Omega(p^{e \beta_X})$.
\end{proof}

\begin{rem}
If $X$ is a $d$-dimensional $F$-finite algebraic variety over an algbraically closed  field $k$, then $\beta_X = d + \log_p [{}^1k:k]$.	
Hence the inequality $\h_t(F_*) \ge d \log p + \log[{}^1k:k]$ holds.
\end{rem}

\section{Upper bounds}

In this section, we give an upper bound of the categorical entropy of the Frobenius pushforward and complete the proof of Theorem \ref{main}.
We begin with providing an easy lemma.

\begin{lem}\label{ext}
Let $R$ be a local ring, and let $x$ be an $R$-regular element. 
Then, for all positive integers $n$, one has that $R/x^nR \in (R/xR)^{* n}$. 	
\end{lem}

\begin{proof}
We use induction on $n$.
The case $n=1$ is clear.
Let $n\ge2$.
As $x$ is $R$-regular, there is an exact sequence $0 \to R/xR \to R/x^nR \to R/x^{n-1}R \to 0$.
Using the induction hypothesis, we get $R/x^nR \in R/xR * R/x^{n-1}R \subseteq R/xR * (R/xR)^{*(n-1)} = (R/xR)^{* n}$. 
\end{proof}

Let $R$ be a local ring.
Let $M$ be a finitely generated module $M$.
Then we denote by $\mu_R(M)$ the minimal number of generators of $M$.
For an integer $n\ge0$, we denote by $\Omega_R^nM$ the $n$th syzygy of $M$ in the minimal free resolution of $M$ and by $\beta_n^R(M)$ the $n$th Betti number of $M$.
Now we are ready to give a proof of the main result of this section.

\begin{thm}\label{ubd}
Let $(R,\fm,k)$ be a $d$-dimensional $F$-finite local ring of characteristic $p$.
Then:
\begin{enumerate}[\rm(1)]
\item
There is an equality $\delta_t(G, {}^eG) = O(([{}^1k:k]p^d)^e)$ for any split generator $G$ of $\db(R)$.
\item
There is an inequality $\h_t(F_*) \le d \log p + \log[{}^1k:k]$.
\end{enumerate}
\end{thm}

\begin{proof}
It is enough to prove (1) because it implies (2).
By Remark \ref{rem}(2), we have only to show that there exists a split generator $G$ such that $\delta_t(G, {}^eG) = O(([{}^1k:k]p^d)^e)$.
Let us show it by induction on $d$.
Assume $d = 0$.
Then $k$ is a split generator of $\db( R)$.
Since ${}^ek \cong k^{\oplus [{}^1k:k]^e}$, we have $\delta_t(k, {}^ek) \le [{}^1k:k]^e =  p^{e\alpha}$, and this implies $\delta_t(k, {}^ek) = O([{}^1k:k]^e)$.

Now suppose that $d > 0$ and that the equality $\delta_t(G, {}^eG) = O(([{}^1k:k]p^{d'})^e)$ holds for any $F$-finite local ring $R$ of dimension $d' <d$ and some split generator $G$ of $\db( R)$.
Let $R$ be a $d$-dimensional $F$-finite local ring.
Take a filtration $0 = I_0 \subseteq I_1 \subseteq \cdots \subseteq I_r = R$ of ideals with $I_i/I_{i-1} \cong R/\fp_i$ for some $\fp_i \in \Spec R$.
For a complex $X \in \db( R)$, there is a filtration $0 = I_0X \subseteq I_1X \subseteq \cdots \subseteq I_rX =X$ of complexes such that $I_iX/I_{i-1}X$ is a bounded complex of finitely generated $R/\fp_i$-modules. 
If each $\db(R/\fp_i)$ has a split generator $G_i$, then the above filtration shows that $\bigoplus_{i=1}^r G_i$ is a split generator of $\db(R)$.
Moreover, one has the inequalities
\begin{align*}
\textstyle\delta_t^{\db(R)}(\bigoplus_{i=1}^r G_i, {}^e(\bigoplus_{i=1}^r G_i))  
&\textstyle\le \sum_{i=1}^r \delta_t^{\db(R)}(\bigoplus_{i=1}^r G_i, {}^e G_i) \\
&\textstyle\le \sum_{i=1}^r \delta_t^{\db(R)}(G_i, {}^e G_i)
\le \sum_{i=1}^r \delta_t^{\db(R/\fp_r)}(G_i, {}^e G_i),
\end{align*}
where the first, second, and third inequalities follow from Lemma \ref{compl}(5), (3) and (6), respectively.
Thus, it suffices to show the statement in the case where $R$ is an integral domain.

By virtue of \cite[Theorem 5.3]{IT}, there exists a nonzero element $x$ of $R$ such that $x\Ext_R^{2d+1}(M,N)= 0$ for all $R$-modules $M$ and $N$.
Let $G$ be a split generator of $\db( R/xR)$ of which $R/xR$ is a direct summand.
Then $\widetilde{G} := G \oplus R$ is a split generator of $\db( R)$.
Indeed, let $M$ be a finitely generated $R$-module $M$.
As $x$ kills $\Ext_R^1(\Omega_R^{2d}M, -) \cong \Ext_R^{2d+1}(M,-)$, the module $M$ is a direct summand of $\Omega_R(\Omega^{2d}_R M/x\Omega^{2d}_R M)$ by \cite[Lemma 2.2]{HP}.
Then
$$
\Omega^{2d}_R M/x\Omega^{2d}_R M \in \db( R/xR) = \thick_{\db( R/xR)}(G) \subseteq \thick_{\db( R)}(\widetilde{G}).
$$
There is an exact sequence $0 \to \Omega_R(\Omega^{2d}_R M/x\Omega^{2d}_R M) \to P \to \Omega^{2d}_R M/x\Omega^{2d}_R M  \to 0$ with $P$ free, and we see that $M \in \thick_{\db( R)}(\widetilde{G})$.
It follows that $\widetilde{G}$ is a split generator of $\db(R)$.

Using Lemma \ref{compl}(2), we have an inequality $\delta_t^{\db( R)}(\widetilde{G}, {}^e\widetilde{G}) \le \delta_t^{\db( R)}(\widetilde{G}, {}^eG) + \delta_t^{\db( R)}(\widetilde{G}, {}^e\!R)$.
Also, by Lemma \ref{compl}(3)(6) and the induction hypothesis, we get
$$
\delta_t^{\db( R)}(\widetilde{G}, {}^eG) \le \delta_t^{\db( R)}(G, {}^eG) \le \delta_t^{\db( R/xR)}(G, {}^eG) = O(([{}^1k:k]p^{e-1})^e).
$$
Therefore, it is enough to prove that $\delta_t^{\db( R)}(\widetilde{G}, {}^e\!R) = O(([{}^1k:k]p^d)^e)$.

The minimal free resolution of the $R$-module ${}^e \!R$ gives rise to an exact sequence 
$$
0 \to \Omega_R^{2d}({}^e\!R) \to R^{\oplus \beta_{2d-1}^R({}^e\!R)} \to \cdots \to R^{\oplus \beta_{0}^R({}^e\!R)} \to {}^e\!R \to 0.
$$
As $x$ is regular on $R$ and ${}^eR$, we obtain an exact sequence
$$
0 \to \Omega_R^{2d}({}^e\!R)/x\Omega_R^{2d}({}^e\!R) \to (R/xR)^{\oplus \beta_{2d-1}^R({}^e\!R)} \to \cdots \to (R/xR)^{\oplus \beta_{0}^R({}^e\!R)} \to {}^e\!R/x\,{}^e\!R \to 0.
$$
Also, there is a canonical short exact sequence
$$
0 \to \Omega_R(\Omega^{2d}_R ({}^e\!R)/x\Omega^{2d}_R ({}^e\!R)) \to R^{\oplus \beta_{2d}^R({}^e\!R)} \to \Omega_R^{2d}({}^e\!R)/x\Omega_R^{2d}({}^e\!R) \to 0.
$$
Here, we use the equalities $\mu_R(\Omega^{2d}_R ({}^e\!R)/x\Omega^{2d}_R ({}^e\!R))) = \mu_R(\Omega^{2d}_R ({}^e\!R)) = \beta_{2d}^R({}^e\!R)$.
As we have already seen, ${}^eR$ is a direct summand of $\Omega_R(\Omega^{2d}_R ({}^e\!R)/x\Omega^{2d}_R ({}^e\!R))$.
Hence there exists a finitely generated $R$-module $M$ such that the following containment holds true.
$$
{}^e\!R \oplus M \in ({}^e\!R/x\,{}^eR)[-2d+1] * (R/xR)^{\oplus \beta_{0}^R({}^e\!R)}[-2d] * \cdots * (R/xR)^{\oplus \beta_{2d-1}^R({}^e\!R)}[-1] * R^{\oplus \beta_{2d}^R({}^e\!R)}.
$$
Since $\widetilde{G}$ contains $R$ and $R/xR$ as direct summands, this yields
\begin{align*}
\delta_t^{\db( R)}(\widetilde{G}, {}^e\!R)
&\textstyle\le \delta_t^{\db( R)}(\widetilde{G}, {}^e\!R/x\,{}^e\!R)\e^{(-2d+1)t} + \sum_{i=0}^{2d} \delta_t^{\db( R)}(\widetilde{G}, (R/xR)^{\oplus \beta_{i}^R({}^e\!R)})\e^{(-2d+i)t}\\
&\textstyle\le \delta_t^{\db( R)}(\widetilde{G}, {}^e\!R/x\,{}^e\!R)\e^{(-2d+1)t} + \sum_{i=0}^{2d} \beta_{i}^R({}^e\!R)\e^{(-2d+i)t}
\end{align*}
It follows from Lemma \ref{ext} that ${}^e\!R/x\,{}^e\!R = {}^e(R/x^{p^e}R)\in{}^e((R/xR)^{* p^e}) \subseteq ({}^e(R/xR))^{* p^e}$.
Hence,
\begin{align*}
\delta_t^{\db( R)}(\widetilde{G}, {}^e\!R/x\,{}^e\!R)
&\le p^e\delta_t^{\db( R)}(\widetilde{G}, {}^e(R/xR)) 
\le p^e\delta_t^{\db( R/xR)}(G, {}^e(R/xR))\\
&\le p^e \delta_t^{\db( R/xR)}(G, {}^eG)
= p^e O(([{}^1k:k]p^{d-1})^e)
= O(([{}^1k:k]p^d)^e).
\end{align*}
Here, the first inequality follows by Lemma \ref{compl}(5), the second by Lemma \ref{compl}(3)(6), the third by Lemma \ref{compl}(2), and the last by the induction hypothesis.
On the other hand, it is shown by \cite[Theorem]{Sei} that $\beta_{i}^R({}^eR) = O(([{}^1k:k]p^d)^e)$.
Consequently, we obtain 
\begin{align*}
\delta_t^{\db( R)}(\widetilde{G}, {}^e\widetilde{G}) =O(([{}^1k:k]p^{d-1})^e) + O(([{}^1k:k]p^d)^e) = O(([{}^1k:k]p^d)^e).
\end{align*}
The proof of the theorem is now completed.
\end{proof}

The combination of Corollary \ref{corlbd}(1) with Theorem \ref{ubd} yields the following result.

\begin{cor}\label{3}
Let $R$ be a $d$-dimensional $F$-finite local ring with characteristic $p$.
Then for any split generator $G$ of $\db( R)$, one has
$
\delta_t(G, {}^eG) = \Theta(([{}^1k:k]p^d)^e).
$
In particular, the equality $\h_t(F_*) = d \log p + \log[{}^1k:k]$ holds.	
\end{cor}



\begin{thebibliography}{99}
\bibitem{DHKK}
{\sc G. Dimitrov, F. Haiden, L. Katzarkov, M. Kontsevich}, Dynamical systems and categories, {\em Contemp. Math.} {\bf 621} (2014), 133--170.
\bibitem{ELS}
{\sc A. Elagin, V. A. Lunts, and O. M. Schn\"urer}, Smoothness of derived categories of algebras, {\em Mosc. Math. J.} {\bf 20} (2020), no. 2, 277--309.
\bibitem{FFO}
{\sc Y.-W. Fan, L. Fu, and G. Ouchi}, Categorical polynomial entropy, \texttt{arXiv:2003:1422}.
\bibitem{HP}
{\sc J. Herzog and D. Popescu}, Thom-Sebastiani problems for maximal Cohen-Macaualy modules, {\em Math. Ann.} {\bf 309} (1997), 677--700.
\bibitem{HH}
{\sc M. Hochster and C. Huneke}, Indecomposable canonical modules and connectedness, {\em Contemp. Math.} {\bf 159} (1994), 197--208.
\bibitem{IT}
{\sc S. B. Iyengar and R. Takahashi}, Annihilation of cohomology and strong generation of module categories, {\em Int. Math. Res. Not. IMRN} {\bf 2} (2016), 499--535.
\bibitem{IT2}
{\sc S. B. Iyengar and R. Takahashi}, Openness of the regular locus and generators for module categories, {\em Acta Math. Vietnam.} {\bf 44} (2019), 207--212.
\bibitem{Kunz}
{\sc E. Kunz}, On Noetherian rings of characteristic $p$, {\em Amer. J. Math.} {\bf 98} (1976), 999--1013.
\bibitem{MZM}
{\sc M. Majidi-Zolbanin and N. Miasnikov}, Entropy in the category of perfect complexes with cohomology of finite length, {\em J. Pure Appl. Algebra} {\bf 223} (2019), no. 6, 2585--2597.
\bibitem{MZMS}
{\sc M. Majidi-Zolbanin and N. Miasnikov, and L. Szpiro}, Entropy and flatness in local algebraic dynamics, {\em Publ. Mat.} {\bf 57} (2013), 509--544.
\bibitem{Mat}
{\sc H. Matsui}, Prime thick subcategories and spectra of derived and singularity categories of noetherian schemes, {\em Pacific J. Math.} {\bf 313} (2021), no. 2, 433--457.
\bibitem{Orl}
{\sc D. Orlov}, Formal completions and idempotent completions of triangulated categories of singularities, {\em Adv. Math.} {\bf 226} (2011), 206--217.
\bibitem{Sei}
{\sc G. Seibert}, Complexes with homology of finite length and Frobenius functors, {\em J. Algebra} {\bf 125} (1989), 278--287.
\end{thebibliography}
\end{document}